\documentclass[11pt]{amsart}
\usepackage{amsmath,amssymb}
\usepackage{geometry}                
\usepackage{graphicx}
\usepackage{amssymb}
\usepackage{epstopdf}
\newtheorem{problem}{Problem}
\newtheorem*{problem*}{Problem}
\newtheorem{theorem}{Theorem}

\newcommand{\bsq}{\boldsymbol q}
\newcommand{\bsp}{\boldsymbol p}
\newcommand{\bsal}{\boldsymbol \alpha}
\newcommand{\bsla}{\boldsymbol \lambda}

\newcommand{\scp}[2]{{\left\langle {#1}\, , \, {#2}\right\rangle}}

\graphicspath{{./FIGURES/}}

\title{LDDMM Surface Registration with Atrophy Constraints}

\begin{document}
\maketitle
\begin{abstract}
Diffeomorphic registration using optimal control on the diffeomorphism group and on shape spaces has become widely used since the development of the Large Deformation Diffeomorphic Metric Mapping (LDDMM) algorithm.  More recently, a series of algorithms involving sub-riemannian constraints have been introduced, in which the velocity fields that control the shapes in the LDDMM framework are constrained in accordance with a specific deformation model. Here, we extend this setting by considering, for the first time, inequality constraints, in order to estimate surface deformations that only allow for atrophy, introducing for this purpose an algorithm that uses the augmented lagrangian method. We also provide a version of our approach that uses a weaker constraint in which only the total volume is forced to decrease. These developments  are illustrated by numerical experiments on brain data.  
\end{abstract}
\section{Introduction}
Over the last couple of decades, multiple studies have provided evidence of anatomical differences between control groups and cognitively impaired groups at the population level, for a collection of diseases, including schizophrenia, depression, Huntington's or dementia \cite{oishi2011multi,faria2011quantitative,cavedo2011local,poulin2011amygdala,younes2012regionally,tosun2013neuroimaging,apostolova20103d,lorenzi20104d,tang2014baseline,younes2014inferring,miller2014amygdalar}. In the particular case of neuro-degenerative diseases, a repeated objective has been to design anatomical biomarkers, measurable from imaging data, that would allow for individualized detection and prediction of the disease. This goal has become even more relevant with the recent emergence of longitudinal studies, involving patients at early stages or ``converters'' which showed that, when the effect is measured at the population level, anatomical changes caused by diseases like Alzheimer's or Huntington's were happening several years before cognitive impairment can be detected on individual subjects. 

Surface registration \cite{vaillant2005currents} using the LDDMM algorithm is a powerful tool for the analysis of shape variation in ROIs represented by triangulated surfaces. While one its main advantages is its flexibility and its ability to render smooth, diffeomorphic, free-form shape changes, there are situations in which prior information is available, and should be used to enhance the results of the shape analysis. In this paper, we focus on situations in which no tissue growth is expected to occur (like with brain longitudinal data). In such contexts, it is natural to ensure that shape analysis should only detect atrophy, even when noise and inaccuracy in the ROI segmentation process may lead in the other direction. (Here, we mean ``atrophy'' in the general sense of local volume loss.) In this paper, we introduce an {\em atrophy-constrained} registration algorithm, that include some of the ideas introduced in \cite{arguillere2014shape}, while extending them to inequality constraints associated to the problem we consider. This algorithm will be described in section \ref{sec:at.lddmm}, with our numerical approach discussed in section \ref{sec:num}. Some theoretical results on existence of solutions and consistency of discrete approximations are provided in section \ref{sec:theorem}. An extension of the algorithm to include affine alignment and its restriction to scalar volume decrease constraint are provided in section \ref{sec:affine} and \ref{sec:volume}. Finally, experimental results are provided in section \ref{sec:exp}.

\section{Atrophy-constrained LDDMM}
\label{sec:at.lddmm}

\subsection{Continuous Optimal Control Problem} 
 The LDDMM algorithm implements an ``optimal control'' strategy in which a template surface $S_0$ is ``driven'' toward a target surface $S_1$ via a time-dependent process $t\mapsto S(t)$, with $S(0)=S_0$. This is achieved by minimizing
\begin{equation}
\label{eq:lddmm.1}
\frac12 \int_0^1 \|v(t)\|_V^2 dt+D(S(1),S_1 )
\end{equation}
subject to the state equation $dS/dt=v(t,S(t))$  where $v$ is a smooth velocity field on $\mathbb R^3$. By $\|v\|_V^2$, we mean a functional norm in a reproducing kernel Hilbert space (RKHS) $V$, that we will assume to be embedded in $C^p_0(\mathbb R^d, \mathbb R^d)$ (the completion, for the standard supremum norm of up to $p$ derivatives, of the space of compactly-supported infinitely differentiable vector fields) with $p\geq 1$. This space can, for example, be defined as the Hilbert completion of 
\[
\|v\|_V^2 =\int_{\mathbb R^3} (Av(x))^T v(x)dx
\]
(originally defined for smooth vector fields),
where $A$ is a differential operator. More generally, letting $\mathbb A: V \mapsto V^*$ be the Hilbert duality mapping, and $\mathbb K = \mathbb A^{-1}$, the
 reproducing kernel of $V$, also denoted $\mathbb K$, is a mapping defined on $\mathbb R^3 \times \mathbb R^3$, with values in $\mathcal M_3(\mathbb R)$ (the set of 3 by 3 real matrices) such that, letting $\mathbb K^i(x,y)$ denote the $i$th column of $\mathbb K(x,y)$, the vector field  $\mathbb K^i(\cdot, y) : x \mapsto \mathbb K^i(x,y)$ belongs to $V$ for all  $y\in \mathbb R^3$ with, for all $v\in V$,
 $\scp{\mathbb K^i(\cdot, y)}{v}_V = v_i(y)$, the $i$th coordinate of $v(y)$. To simplify the notation, we will assume in this paper that $\mathbb K$ is a scalar kernel, i.e., that it takes the form $\mathbb K(x,y) = K(x,y)\, \mathrm{Id}_{\mathbb R^3}$ where $K$ is-scalar valued.

The function $D$ in \eqref{eq:lddmm.1} is a measure of discrepancy that penalizes the difference between the controlled surface $S(\cdot)$ at the end of its evolution and the target surface $S_1$.  
Among the measures introduced in the literature in combination with the LDDMM algorithm, the most convenient computationally are designed as Hilbert space norms between surfaces considered as linear forms over spaces of smooth structures. The simplest example,  linear forms on smooth scalar functions arising from integrating functions over surfaces, yields the ``measure matching'' cost introduced in \cite{glaunes2004diffeomorphic,glaunes2006modeling}).  ``Current matching'', introduced in \cite{glaunes2008large,vaillant2005currents} results from integrating smooth differential forms over oriented surfaces.  More recently varifold-based cost functions \cite{charon2013varifold} were designed, in which functions defined on $\mathbb R^3 \times \mathit{Gr}(2,\mathbb R^3)$ (the Grassmannian manifold of 2D spaces in $\mathbb R^3$) are integrated over the surface. Details on these cost functions, their discrete versions on triangulated surfaces, and the computation of their gradient  are provided in the cited references.

 In optimal control language, $v$ is the ``control'',  $S$ is the ``state'', and $v$ is optimized in order to bring the state close to a desired endpoint. With this construction, each point $x_0$ in $S_0$ is registered to a point $x(t)= \varphi(t,x_0 )$ in $S(t)$ that evolves according to the differential equation $dx/dt=v(t,x)$, with $x(0)=x_0$. The overall evolution is diffeomorphic, i.e., for each time t, $\varphi(t,\cdot)$ can be extended to a smooth invertible transformation with smooth inverse on $\mathbb R^3$. 

To define our atrophy constraints, we assume that surfaces are closed and oriented. We let $N_0(t, x_0)$ be the outward-pointing unit normal to $S_0$. An outward-pointing normal to $S(t)$ at $x= \varphi(t,x_0)$ is then given by
\begin{equation}
\label{eq:norm.0}
N(t, x) = d\varphi(t,x_0)^{-T}N_0(t, x_0)
\end{equation}
where $d\varphi(t, x_0)$ denote the differential of $y \mapsto \varphi(t, y)$ at $y=x_0$ (a 3 by 3 matrix), with the $-T$ exponent indicating the inverse transposed ($N(t, \cdot)$ does not necessarily have norm one).

We will express our atrophy constraint by the fact that the surface evolves inward at all points, i.e., by
$
v(t,x)^TN(t,x) \leq 0
$
for all $x\in S(t)$ and  $t\in [0,1]$. Adding this constraint to the original surface-matching LDDMM problem leads to the atrophy-constrained problem 
\[
\frac12 \int_0^1 \|v(t)\|_V^2 dt+D(S(1),S_1 )\to \min
\]
subject to  $\partial _t S=v(t,S(t))$, $v(t,x)^TN(t,x) \leq 0$.\\

We now reformulate this problem under the assumption that $S_0$ is parametrized with an embedding  $q_0: M\to \mathbb R^3$, where $M$ is a two-dimensional Riemannian manifold. This is no loss of generality, since one can always take $M=S_0$ and $q_0 = \mathrm{identity}$. We take parametrizations $q: M\to \mathbb R^3$ as state variables, together with functions $N: M \to \mathbb R^3$ and solve
\begin{problem}
\label{prob:1}
Minimize
\begin{equation}
\label{eq:prob1.1}
\frac12 \int_0^1 \|v(t)\|_V^2 dt+D(q(1,M),S_1 ),
\text{ subject to } 
\begin{cases}
q(0, \cdot) = q_0, N(0, \cdot) = N_0,\\
\partial_t q(t, \cdot) = v(t,q(t, \cdot)),\\
 \partial_t N(t, \cdot) = - dv(t,q(t, \cdot))^T N(t, \cdot),\\
  v(t, q(t,\cdot))^T N(t, \cdot) \leq 0
\end{cases}
\end{equation}
\end{problem}
where, with a slight change of notation, $N_0(m)$ is the outward-pointing unit normal to $S_0$ at $q_0(m)$. $N(t,m)$ is then an outward-pointing (not necessarily unit) normal to $S(t) = q(t,M) = \varphi(t, S_0)$ at $q(t,m)$.  The third equation in the constraints is the time derivative of \eqref{eq:norm.0}.

\subsection{Discrete Approximations}
We now assume that surfaces are triangulated, identifying $S$ with a pair $(\boldsymbol q, F)$, where $\boldsymbol q \in (\mathbb R^3)^n$ is a set of  $ n$ vertices (where $n$ depends on $S$) and $F\subset \{1,\ldots,n\}^3$ lists the indices of vertices that form the triangular faces. We assume that the surface is oriented, so that an edge which belongs to two faces is oriented in different directions in each face. If $\boldsymbol q = (q_1, \ldots, q_n)$, and $f = (i,j,k)\in F$, the area weighted normal to $f$ is
\begin{equation}
\label{eq:n1}
N(\boldsymbol q, f) = \frac12 (q_j - q_i)\times (q_k-q_i),
\end{equation}
which is invariant by circular permutation of $i$, $j$ and $k$. From this we define the area-weighted normal at vertex $q_k$ by
\begin{equation}
\label{eq:n2}
N_k(\boldsymbol q, F) = \sum_{f\in F: k\in f} N(\boldsymbol q, f)
\end{equation}
with $\boldsymbol N(\boldsymbol q, F) = ( N_1(\boldsymbol q, F), \ldots, N_n(\boldsymbol q, F))$.

To define a  discrete version of Problem \ref{prob:1}, we introduce a small relaxation parameter $\varepsilon\geq 0$ and  state variables $\boldsymbol q = (q_1, \ldots, q_n)$ and $\boldsymbol N  = (N_1, \ldots, N_n)$, initialized with a initial surface $S_0 = (\boldsymbol q_0, F_0)$, with a target surface $S_1= (\boldsymbol q_1, F_1)$. The discrete problem will minimize
\[
\frac12 \int_0^1 \|v(t)\|_V^2 dt+D( S(1), S_1 ),
\text{ subject to } 
\begin{cases}
\boldsymbol q(0, \cdot) = \boldsymbol q_0, \boldsymbol N(0) = \boldsymbol N(q_0, F_0),\\
\partial_t q_k(t) = v(t,q_k(t)),\\
 \partial_t N_k(t) = - dv(t,q_k(t))^T N_k(t),\\
  v(t, q_k(t))\cdot N_k(t) \leq \varepsilon |N_k(t)|,\  k = 1, \ldots, n
\end{cases}
\]
with $S(t)  = (\boldsymbol q(t), F_0)$. The introduction of the parameter $\varepsilon$ is justified by Theorem \ref{th:convergence}, and is also natural given discrete approximation errors.\\

Given that the end-point cost and the constraints depends on $v$ only through the  configurations $\boldsymbol q(t)$, one shows, using standard RKHS reductions, that the optimal $v$ takes the form
\[
v(t, \cdot) = \sum_{k=1}^n K(\cdot, q_k(t)) \alpha_k(t)
\]
where $K$ is the reproducing kernel of $V$. Using this, the previous problem reduces to
\begin{problem}
\label{prob:2}
Minimize
\begin{equation}
\label{eq:prob2.1}
\frac12 \int_0^1 \sum_{k,l=1}^n K(q_k(t), q_l(t)) \alpha_k(t)\cdot \alpha_l(t)  dt+D( S(1), S_1 )
\end{equation}
\begin{equation}
\label{eq:prob2.2}
\text{subject to  } \quad
\begin{cases}
\displaystyle
\boldsymbol q(0, \cdot) = \boldsymbol q_0, \boldsymbol N(0) = \boldsymbol N(q_0, F_0),\\
\displaystyle
\partial_t q_k(t) = \sum_{l=1}^n K(q_k(t), q_l(t))\alpha_l(t),\\
\displaystyle
 \partial_t N_k(t) = - \sum_{l=1}^n \partial_1 K(q_k(t), q_l(t)) N_k(t)\cdot \alpha_l(t),\\
  \sum_{l=1}^n K(q_k(t), q_l(t))\alpha_l(t) \cdot N_k(t) \leq \varepsilon |N_k(t)|, k = 1, \ldots, n
\end{cases}
\end{equation}
with $S(t)  = (\boldsymbol q(t), F_0)$.
\end{problem}

However, in the discrete case, it is possible to avoid the introduction of $\boldsymbol N$ as a state variable and solve instead:
\begin{problem}
\label{prob:3}
Minimize
\begin{equation}
\label{eq:prob3.1}
\frac12 \int_0^1 \sum_{k,l=1}^n K(q_k(t), q_l(t)) \alpha_k(t)\cdot \alpha_l(t)  dt+D( S(1), S_1 )
\end{equation}
\begin{equation}
\label{eq:prob3.2}
\text{subject to  } \quad
\begin{cases}
\displaystyle
\boldsymbol q(0, \cdot) = \boldsymbol q_0, \\
\displaystyle
\partial_t q_k(t) = \sum_{l=1}^n K(q_k(t), q_l(t))\alpha_l(t),\\
\displaystyle
 \sum_{l=1}^n K(q_k(t), q_l(t))\alpha_l(t) \cdot N_k(\boldsymbol q(t), F_0) \leq \varepsilon | N_k(\boldsymbol q(t), F_0)|, \quad k = 1, \ldots, n
\end{cases}
\end{equation}
with $S(t)  = (\boldsymbol q(t), F_0)$.
\end{problem}
Note that the apparent simplification is balanced by the fact that the constraint becomes a more complex function of the state and controls, with $N_k(\boldsymbol q, F_0)$ given by \eqref{eq:n1} and \eqref{eq:n2}.

\section{Numerical Method}
\label{sec:num}

Problem \ref{prob:3} is solved using augmented Lagrangian methods, introducing, as described in \cite{nocedal2006numerical}, slack variables to complete inequality constraints. More precisely, let 
\[
C_{kl}(\boldsymbol q) = K(q_k, q_l) N_k(\boldsymbol q, F_0)^T 
\]
and $C (\bsq) = (C_{kl}(\bsq), k,l=1, \ldots,n)$ the associated $n \times 3n$ matrix.
Let $K(\bsq)$ be the $3n\times 3n$ matrix formed from blocks $(K(q_k,q_l)\, \mathrm{Id}_{\mathbb R^3}, k,l=1, \ldots, n)$.  For a vector $u$, let  $u^+$ denote the vector formed with the positive parts of each of the coordinates of $u$. Let $|\boldsymbol N|$ denote the $n$-dimensional vector will $k$th coordinates equal to $|N_k|$.

The augmented Lagrangian is defined by
\begin{multline}
\label{eq:aug.lag}
F(\bsal, \bsla) = \frac12 \int_0^1 \bsal(t)^TK(\bsq(t))\bsal(t) dt + D(S(1), M) \\
+ \frac\mu2 \int_0^1 \left|\left(C(\bsq(t)) \bsal(t) - \varepsilon |\boldsymbol N(t)|-\frac{\bsla(t)}{\mu}\right)^+\right|^2 dt - \frac1{2\mu}\int_0^1 |\bsla(t)|^2dt.
\end{multline}
Here, the parameter $\mu$ is a small positive number, $\bsla \in \mathbb R^n$ is a Lagrange multiplier and $\bsq$ is considered as a function of $\bsal$ via the state equation $\partial_t \bsq = K(\bsq) \bsal$. The augmented Lagrangian optimization procedure iterates the following  steps (starting with initial values $(\bsal_0, \bsla_0)$):
\begin{enumerate}
\item Minimize $\bsal \mapsto F(\bsal, \bsla_n)$ to obtain a new value $\bsal_{n+1}$ (and  $\bsq_{n+1}$ via the state evolution equation).
\item Update $\bsla$, with $\bsla_{n+1} = -  \mu\left(C(\bsq_{n+1}) \bsal_{n+1} - \varepsilon |\boldsymbol N|-\mu\bsla_n\right)^+$.
\item If needed (e.g., if $|(C(\bsq_{n+1}) \bsal_{n+1} - \varepsilon |\boldsymbol N|)^+|$ is larger than a threshold $\delta_n$), replace $\mu$ by a smaller number, $\rho\mu$, with $\rho < 1$.
\end{enumerate}

The most expensive step is, of course, the first one, which requires to solve an optimal control problem similar in complexity to the unconstrained problem. The computation of $\nabla_{\bsal} F(\alpha, \lambda)$ (the gradient of $F$ with respect to $\bsal$) uses a back-propagation algorithm, also called the adjoint method. Since similar computations are described in multiple places \cite{vaillant2005currents,younes2009evolutions,Cotter:2009,azencott2010diffeomorphic}, we briefly summarize it here. 
\begin{enumerate}
\item Given $\bsal$, compute the associated state $\bsq$ via $\partial_t \bsq = K(\bsq)\bsal$ and the matrix $C(\bsq)$.
\item Introduce a co-state $\bsp(t) = (p_1, \ldots, p_n)\in (\mathbb R^3)^n$, $t\in [0,1]$, defined by \\$\bsp(1) =  \nabla_q D((\bsq(1),F_0), S_1)$ and 
\[
\partial_t \bsp = - \nabla_{\bsq} \left(\bsp^TK(\bsq)\bsal -\frac12\bsal^TK(\bsq)\bsal - \frac\mu2 \left|\left(C(\bsq)\bsal- \varepsilon |\boldsymbol N|- \frac\bsla\mu\right)^+\right|^2\right)
\]
\item Define
$
\displaystyle
\nabla_{\bsal} F = K(\bsq)(\bsal - \bsp)+ \mu \left(\left(C(\bsq)\bsal- \varepsilon |\boldsymbol N| -\frac\bsla\mu\right)^+\right)^T C(\bsq).
$
\end{enumerate}

\section{Existence of Solutions and Convergence}
\label{sec:theorem}
\begin{theorem}
\label{th:existence}
Assume that $V$ is continuously embedded in $C^p_0(\mathbb R^3, \mathbb R^3)$ for $p \geq 2$ and that the data attachment term $D$ is such that 
$\varphi \mapsto D(\varphi(S_0), S_1)$,
defined over all $C^p$-diffeomorphisms $\varphi$ such that $\varphi-\mathit{id}\in C^p_0$, is continuous for the uniform $C^{p}$ convergence over compact sets. Then Problems \ref{prob:1}, \ref{prob:2} and \ref{prob:3} always have an optimal solution $v\in L^2(0,1, V)$. 
\end{theorem}
\begin{proof}
The theorem is proved along the same lines as similar statements addressing the existence of solutions for LDDMM problems  \cite{trouve1995action,Dupuis1998,younes2010shapes,younes2012constrained,arguillere2014shape}. Considering Problem \ref{prob:1}, and starting with a minimizing sequence $v^n$, the boundedness of $v^n$ in $L^2(0,1; V)$ implies that (extracting a subsequence if needed) $v^n$ converges weakly in this Hilbert space to a limit $v$ (with a norm in $L^2(0,1;V)$ no larger than the $\liminf$ of the norms of $v^n$). This, in turn, implies that the associated flows of diffeomorphisms $\varphi^n$ (associated to $v^n$) and their first $p$ spatial derivatives, converge uniformly (in time and space) over compact sets to $\varphi$ (associated to $v$) and its first $p$ derivatives. One gets from this that $D(\varphi^n(1,S_0), S_1)$ converges to $D(\varphi(S_0), S_1)$. Letting $(q^n, N^n)$ denote the state in Problem 1 associated to the control $v^n$, we  have $q^n(t) = \varphi^n(t) \circ q_0\to q$. This implies that the cost function is minimized at $v$.  Note that $N^n(t) = d\varphi^n(t)^{-T}N_0$ and therefore converges to $N(t)$. 

We now show that $v(t, q(t, x))^TN(t,x)\leq 0$ for all $x\in M$ and almost all $t\in [0,1]$. Fixing $x$, one easily deduces from the facts that $v^n$ converges weakly to $v$, $dv^n$ is uniformly bounded, $q^n$ and $N^n$ converge uniformly to $q$ and $N$, that 
$v^n(t, q^n(t, x))^TN^n(t,x)$ converge weakly to $v(t, q(t, x))^TN(t,x)$ in $L^2(0,1;\mathbb R)$, which implies that $v(t, q(t, x))^TN(t,x) \leq 0$ for almost all $t$ (the set of non-positive a.e functions is a closed convex set in $L^2$ and therefore also weakly convex). By considering a countable dense set of $x's$, and using the fact that $x \mapsto v(t, q(t, x))^TN(t,x)$ is continuous, we find that  $v(t, q(t, x))^TN(t,x) \leq 0$ for all $x\in M$ and almost all $t\in[0,1]$.
The same argument can be used to prove existence of solutions for Problems \ref{prob:2} and \ref{prob:3}. 
\end{proof}

Next, we adress the problem of convergence as the triangulation is refined. More precisely, as the triangulation of a fixed initial surface $S_0$ gets finer and finer, do solutions of Problems \ref{prob:2} and \ref{prob:3} converge to solutions of Problem \ref{prob:1}. We will see that we need to slightly relax the constraints. 
\begin{theorem}
\label{th:convergence}
Assume that $V$ is continuously embedded in $C^p_0(\mathbb R^3, \mathbb R^3)$ for $p \geq 2$, with $\Vert v\Vert_V\leq c\max_{x\in \mathbb R^3}(\vert v(x)\vert,\vert dv(x)\vert)$ where $\vert v(x)\vert$ is the Euclidean norm of $v(x)$ and $\vert dv(x)\vert$ is the operator norm of $dv(x)$ and $c>0$ is fixed. Also assume that if a sequence of triangulations $(S^n)_{n\in\mathbb N}$ converges to a surface $S$ in the sense of currents, then $D(S^n,S_1)\rightarrow D(S,S_1)$.

Let $(S^n_0)_{n\in\mathbb N}=({\bf q}_0^n,T^n)_{n\geq 3}$ be an increasingly fine triangulation with $n$ points of a compact surface $S_0$  of class $\mathcal{C}^2$ with normal vector field $N$. We assume $q_{k,0}^{n_1}=q_{k,0}^{n_2}$ for any $k\leq\min(n_1,n_2)$ (so that $q_k,\ k\in\mathbb N$ can be defined independently of $n$), that the averaged normals $N^n_{k,0}$ at $q_{k,0}^n$ converge to $N(q_{k,0}^n)$ after normalization as $n$ goes to infinity, and that $\{q_k,\ k\in\mathbb N\}$ is dense in $S_0$. Then, there exists a decreasing sequence or real numbers $\varepsilon^n>0$ 
with $\varepsilon^n\rightarrow0$, such that if  $v^n\in L^2(0,1;V)$ solves Problem \ref{prob:2} with $\varepsilon = \varepsilon^n$, 
then the sequence $(v^n)_{n\in\mathbb N}$ is bounded and any weak limit point $v^*$ of $v^n$ is a solution to Problem \ref{prob:1}.
\end{theorem}
\begin{proof}
First of all, since $\int_0^1\Vert v^n(t)\Vert^2\leq 2 D(S^n_0,S_1)$ and $D(S^n_0,S_1)\rightarrow D(S_0,S_1)$, we see that $(v^n)_{n\in\mathbb N}$ is bounded in $L^2(0,1;V)$ and therefore contained in a weak compact subset. We can assume that $(v^n)_{n\in\mathbb N}$ is weakly convergent without loss of generality. Let $v^*$ be the weak limit of $(v^n)$. We start by proving that $v^*$ satisfies the constraints of Problem \ref{prob:1}. Let $(\bsq^n, \boldsymbol N^n)$ denote the states for the discrete problem.
Also let $q_k^*(t)$ satisfy $\partial_tq_k^*(t)=v^*(t,q_k(t))$ with $q^*_k(0)=q_{0,k}$, $\partial_tS^*(t)=v^*(t,S^*(t))$ with $S^*(0)=S^0$, and $N_u^*(t)$ be the unit normal vector field to $S^*(t)$. 

Then, using the same method as in the previous theorem, we get that $\partial_tq_k^n=v(\cdot,q_k^n(\cdot))$ converges weakly to $\partial_tq_k^*$ in $L^2$ and that $N_k^n(\cdot)/\vert N_k^n(\cdot)\vert$ converges strongly to $N_u^*(t,q^*_k(t))$ in $\mathcal C^0$. As a direct consequence, for every $f:[0,1]\rightarrow\mathbb R$ smooth nonnegative function, we easily obtain $\int_0^1\partial_tq^*_k(t)\cdot N_u^*(t,q^*_k(t))f(t)dt\leq 0$,
so that $\partial_tq^*_k(t)\cdot N_u^*(t,q^*_k(t))=v^*(t,q^*_k(t))\cdot N_u^*(t,q^*_k(t))\leq 0$ for almost every $t$ and every $k$. Therefore, $v^*$  satisfies the constraints of Problem \ref{prob:1}. 

We still need to prove that $v^*$ is optimal. Let $v$ be a solution of Problem \ref{prob:1}. Then $t\mapsto \Vert v(t)\Vert$ is constant and no greater than $\sqrt{2D(S_0,S_1)}$. Let $\varphi$ be the flow of $v$. Using Gronwall's estimates, one obtains a sequence $\varepsilon^n \to 0$ such that $v(t,\varphi(t,q^n_{0,k}))\cdot \frac{d\varphi^{-T}(t,q^n_{0,k})N^n_{0,k}}{{\vert d\varphi^{-T}(t,q^n_{0,k})N^n_{0,k}\vert}}\leq \varepsilon^n$, with $\varepsilon^n$ depending only on $V$, $D(S_0,S_1)$, and the value of $\left\vert N(q^n_{k,0})-N^n_{k,0}/\vert N^n_{k,0} \vert\right\vert$.
Hence, $v$ satisfies the constraints of the discrete problems, so that
 \[
 \frac12 \int_0^1\Vert v^n(t)\Vert_V^2dt +D(S^n(1),S_1)\leq \frac12 \int_0^1\Vert v(t)\Vert_V^2dt +D(\varphi(1,S^n_0),S_1). 
 \]
 Letting $n$ go to infinity, we get the same inequality with $v^*$ and $S^*$ in place of $v^n$ and $S^n$,
so that $v^*$ is optimal.

The value of $\varepsilon^n$ can be explicitely bounded from above with respect to the $\mathcal{C}^2$-norm of a parametrization of $S_0$. Note that this proof can be adapted to the case of Problem \ref{prob:3}, although the exact values of $\varepsilon^n$ may differ. 
\end{proof}

\section{Affine Alignment}
\label{sec:affine}
Because the RKHS $V$ is embedded in $C^p_0(\mathbb R^3, \mathbb R^3)$, vector fields $v\in V$ vanish at infinity. This implies, in particular, that affine transformations do not belong to this Hilbert space, and that the diffeomorphic registration does not incorporate any rigid alignment.  If such an alignment is needed, one can include it in the optimal control framework by completing the control with the corresponding vector fields. 

Let the registration be computed over  $G\ltimes \mathbb R^3$, where $G$ is a subgroup of  $\mathit{GL}_3(\mathbb R)$, $\ltimes$ referring to the semi-direct product extending $G$ with translations to obtain affine transformations. Let $\mathfrak g$ be the Lie algebra of $G$, with basis $E_1, \ldots, E_h$. Instead of $v\in V$, we use a control given by $(v, \beta_1, \ldots, \beta_h, \tau)$ with $\beta_1, \ldots, \beta_h\in \mathbb R$ and $\tau\in \mathbb R^3$, and the state equation
\begin{equation}
\label{eq:aff.1}
\partial_t q_k(t) = v(t, q(t)) + \sum_{l=1}^h \beta_l(t) E_l q(t) + \tau(t)
\end{equation}
with associated cost
$
\frac12 \int_0^1 \left(\|v(t)\|_V^2 + \sum_{k=1}^h c_k\beta_k(t)^2 + c_0 |\tau(t)|^2\right) dt$
for some non-negative numbers $c_0, c_1, \ldots, c_h$. The extension of the numerical procedure to this setting is rather straightforward, and not detailed here. Of course, the affine components must be added to $v$ in the atrophy constraint $v\cdot N \leq 0$.\\

Consider the special case $G = \mathit{SO}_3$  is the rotation group (so that $h=3$) and assume that Euclidean transformations act as isometries on $V$, which means that, for all $v\in V$, $R\in \mathit{SO}_3$, $b\in \mathbb R^3$, the vector field $\tilde v: x \mapsto R^T v(Rx + b)$ also belongs to $V$ and $\|\tilde v\|_V = \|v\|_V$. Euclidean-invariant RKHS's of vector fields are extensively described in \cite{glaunes2014Matrix}, to which we refer for more details. In the case of scalar kernels $\mathbb K(x,y) = K(x,y) \mathit{Id}_{\mathbb R^3}$, Euclidean invariance implies that $K$ is a radial kernel, i.e., that $K(x,y) = \gamma(|x-y|^2)$ for some function $\gamma$ (additional conditions on $\gamma$ are needed to ensure that $K$ is a positive kernel; see \cite{glaunes2014Matrix}). Assume finally that the end-point cost $D$ is invariant by Euclidean transformation: $D(S, S') = D(RS+b, RS'+b)$. In this case, the optimal control problem (using $c_0=\cdots=c_3=0$) that minimizes $\frac12 \int_0^1 \|v(t)\|_V^2 dt + D( S(1), S_1 )$ in $v, \boldsymbol \beta, \tau$, subject to
\begin{equation}
\label{eq:aff1.2}
\begin{cases}
q(0) = q_0, \\
\partial_t q(t) = v(t, q(t)) + \sum_{l=1}^3 \beta_l(t) E_lq(t) + \tau(t)\\
\partial_t N(t, \cdot) = - \left(dv(t,q(t, \cdot)) + \sum_{l=1}^3 \beta_l(t) E_l\right)^T N(t, \cdot)\\
\left(v(t, q(t)) + \sum_{l=1}^3 \beta_l(t) E_lq(t) + \tau(t)\right) \cdot N(\boldsymbol q(t)) \leq 0
\end{cases}
\end{equation}
is equivalent to minimizing 
$\frac12 \int_0^1 \|\tilde v(t)\|_V^2  dt + D(\tilde S(1),  R_1S_1+b_1 )$ in $\tilde v, R_1, b_1$, subject to
\begin{equation}
\label{eq:aff2.2}
\begin{cases}
\displaystyle
\tilde q(0) = q_0, \\
\displaystyle
\partial_t \tilde q(t) = \tilde v(t, \tilde q(t)) \\
\displaystyle
\partial_t \tilde N(t, \cdot) = - d\tilde v(t,\tilde q(t, \cdot))^T \tilde N(t, \cdot)\\
\displaystyle
\tilde v(t, \tilde q(t))  \cdot \tilde N(t,  \tilde q(t)) \leq 0
\end{cases}
\end{equation}
via the change of variable $q(t) = R(t) \tilde q(t) + b(t)$, $v(t, x) = R(t)^{-1} \tilde v(t, Rx+b)$, $N(t) = R(t) \tilde N(t)$, with 
$\partial_t R = \sum_{l=1}^3 \beta_l(t) E_lR$, $\partial_t b = \sum_{l=1}^3 \beta_l(t) E_l b(t) + \tau(t)$,
$R_1 = R(1)^{-1}$ and  $b_1 = -R(1)^{-1}b(1)$.
In other terms, Euclidean registration via optimal control using  \eqref{eq:aff1.2} is equivalent to the original atrophy-constrained LDDMM optimizing its target over an orbit under the action of rigid transformations. 

Note that $c_0 = \cdots = c_h = 0$ should not be used with general affine transformations when non-compact components of the affine group are added to rotations. Intuitively, this would allow small deformations to be scaled up to larger ones at zero cost, and one can conjecture that the associated optimal control problem has no solutions, and admits minimizing sequences of controls with vanishing cost at the limit. The equivalence with a problem in which the target is allowed to vary over its orbit via affine transformations is not true either for groups larger than the Euclidean one, essentially because invariant kernels do not exist in such cases.  \\

Some attention should be paid to  the time discretization, in particular in the rigid case. In our implementation, we use a simple Euler scheme to discretize the equation $\partial_t q = v(t,q)$, i.e., we take
$q(t+\delta t) = q(t) + \delta t\, v(t, q(t))$. 
When using rigid registration, however, we take, with $A(t) = \sum_{l=1}^3 \beta_l(t) E_l$ a skew-symmetric matrix
\[
q(t+\delta t) = e^{\delta t\, A(t)} q(t) + \delta t\, v(t, q(t)) + \delta t\, \tau(t)
\] 
to discretize $\partial_t q = v(t,q) + Aq + \tau$, with the explicit expression 
$e^U = \mathrm{Id} + \frac{\sin c_U}{c_U} U + \frac{1 - \cos c_U}{c_U^2} U^2$,  $c_u = \sqrt{-\mathrm{tr}(U^2)}$
for a 3 by 3 skew-symmetric matrix $U$. This ensures that the rigid registration remains a displacement, even for large values of the coefficients $\beta_l$, which are made possible by the absence of cost on this transformation.

\section{Global Volume Constraint}
\label{sec:volume}
Ensuring that the total volume of the surface decreases (instead of enforcing inward motion at every point) can be done very similarly to the full atrophy constraint, using the single constraint
$\sum_{k=1}^n v(t, q_k(t))\cdot N_k(t)\leq 0$
or, after reduction
$
\sum_{k,l=1}^n K(q_k(t), q_l(t)) \alpha_l(t)\cdot N_k(t) \leq 0$,
where $N_k$ is given by \eqref{eq:n2}. It is important here to use the area-weighted normal,  to discretize the continuous constraint
$\int_{S(t)} v(t, \cdot) \cdot N(t, \cdot) \mathit{vol}_{S(t)} \leq \varepsilon$,
which provides the derivative of the total volume with respect to time (where $\mathit{vol}_{S(t)}$ is the volume form on $S(t)$). This constraint can be rewritten as $\mathbf 1_n^T C(\bsq(t))\bsal(t) \leq \varepsilon$, where $\mathbf 1_n$ is the $n$-dimensional vector with all coordinates equal to 1. The reformulation of the augmented Lagrangian method for this scalar constraint is straightforward and left to the reader.

\section{Experimental Results}
\label{sec:exp}

Fig. \ref{fig:1} provides three examples of segmented hippocampus surfaces taken from the BIOCARD longitudinal study \cite{miller2014amygdalar,younes2014inferring}. Each row corresponds to a different subject,  starts with a baseline image, and then compares the outputs of unconstrained, volume-constrained and atrophy-constrained  surface registration when mapping to a follow-up surface. The color map is proportional to the total normal displacement during the estimated deformation. The segmented follow-up of the first two subjects is slightly larger in volume than the baseline and this is corrected by the two constrained method. The deformation pattern in the volume-constrained approach is however very similar to the unconstrained case (the deformed surface is only slightly, and almost uniformly, smaller in the constrained case, no ensure that no volume is gained). The deformation pattern in the atrophy-constrained case is completely different, since it can only include atrophy (blue). The last subject had significant volume lost, so that the volume-constrained and the unconstrained registrations return identical results, both having a few regions with outward growth, which disappear in the atrophy-constrained case.   
\begin{figure}[h]
\centering
\includegraphics[width=0.15\linewidth]{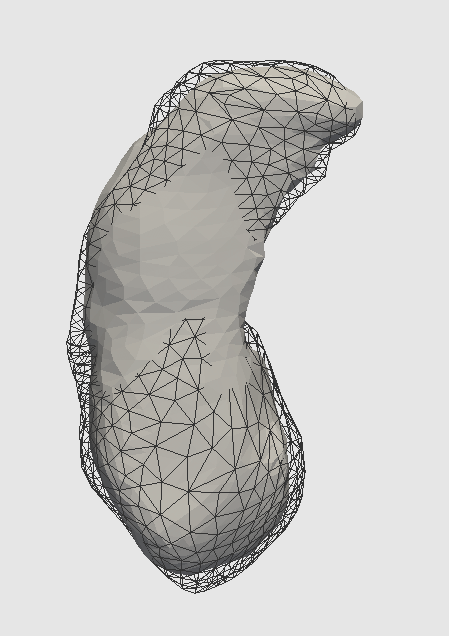}
\includegraphics[width=0.15\linewidth]{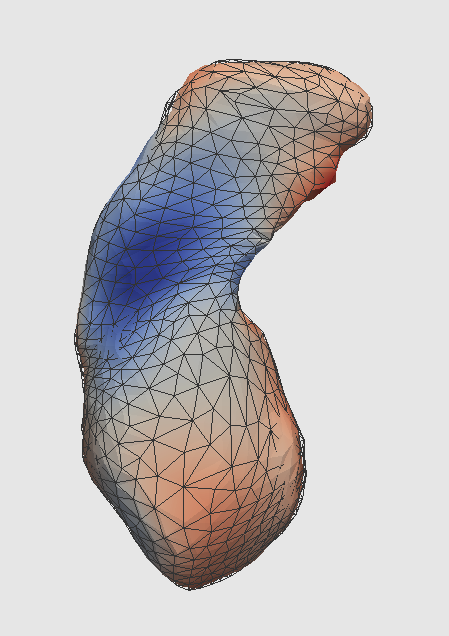}
\includegraphics[width=0.15\linewidth]{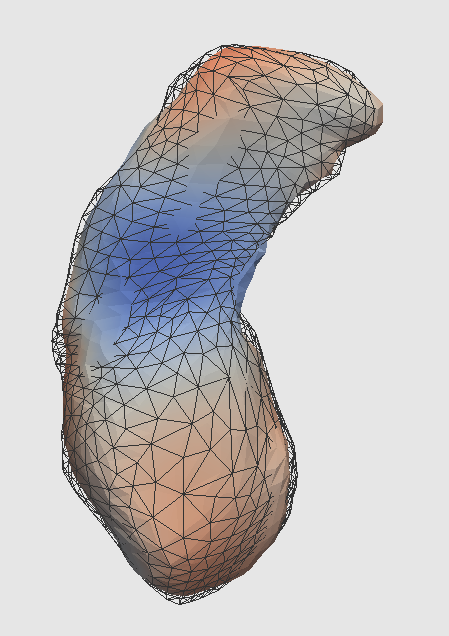}
\includegraphics[width=0.15\linewidth]{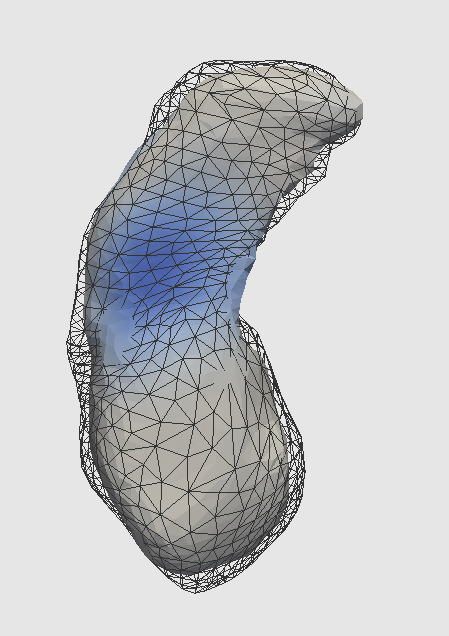}
\includegraphics[trim=6cm 5cm 6cm 5cm, clip,width=0.065\linewidth]{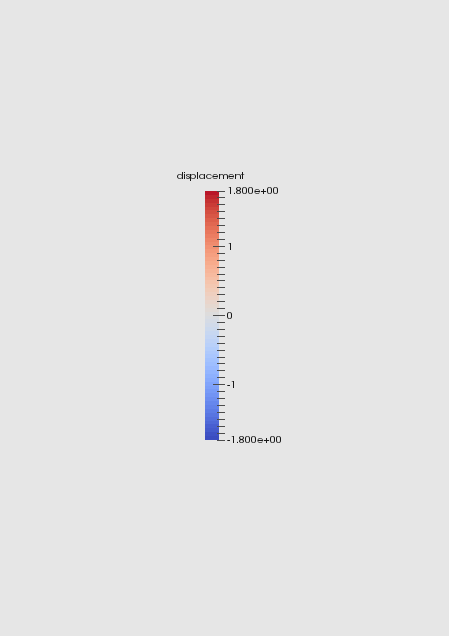}\\
\includegraphics[width=0.15\linewidth]{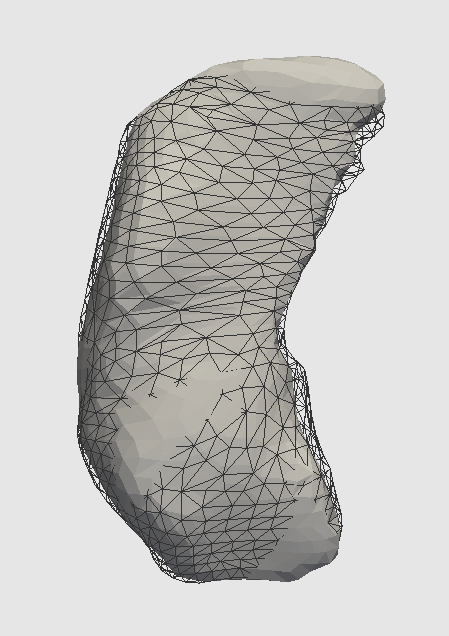}
\includegraphics[width=0.15\linewidth]{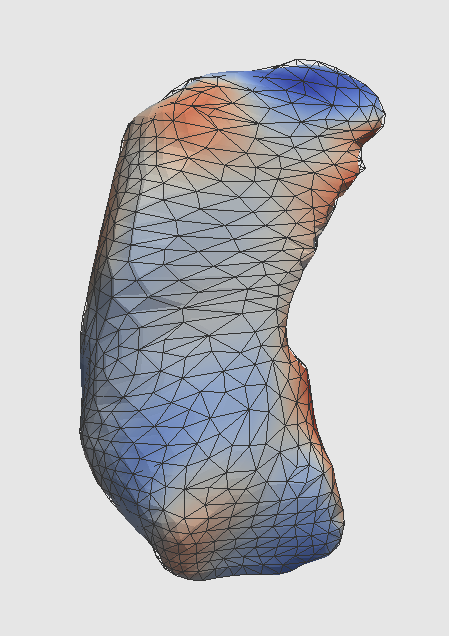}
\includegraphics[width=0.15\linewidth]{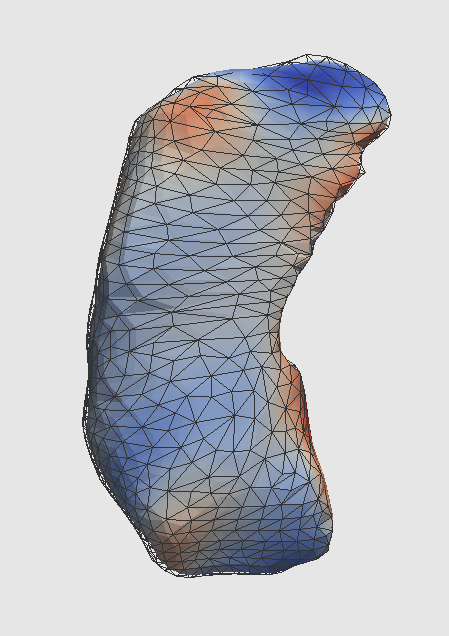}
\includegraphics[width=0.15\linewidth]{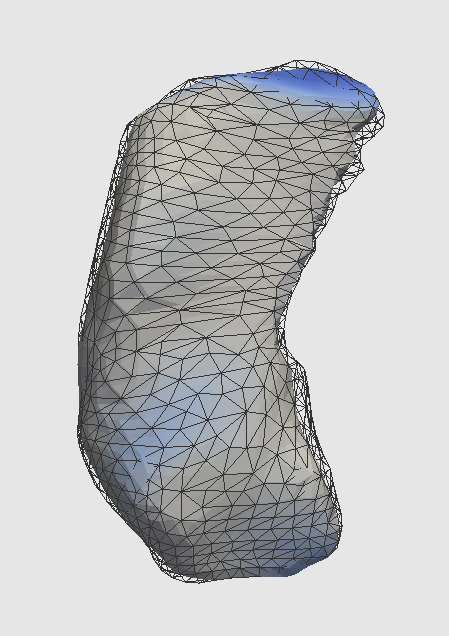}
\includegraphics[trim=6.4cm 5cm 5.6cm 5cm, clip,width=0.065\linewidth]{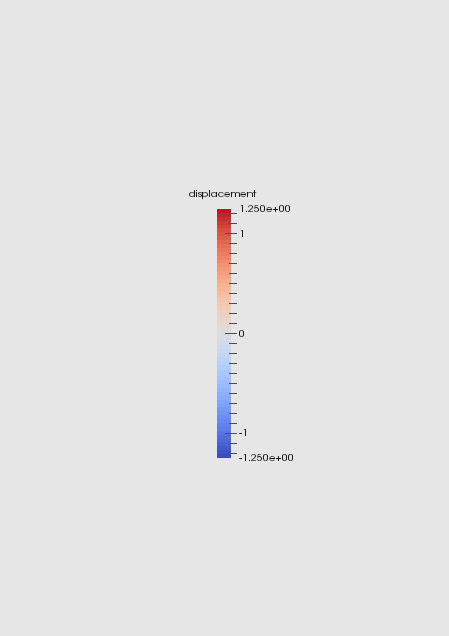}\\
\includegraphics[width=0.15\linewidth]{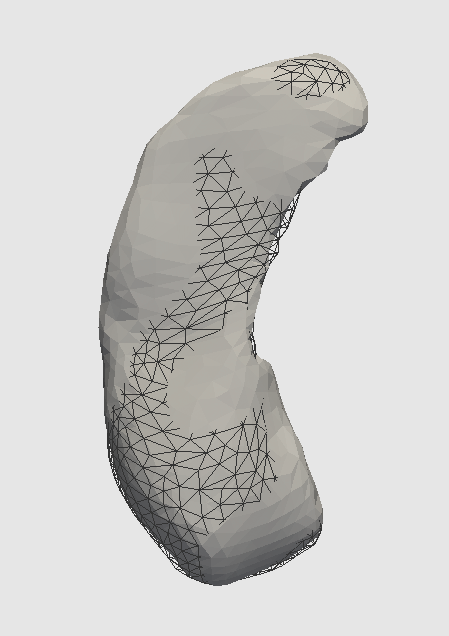}
\includegraphics[width=0.15\linewidth]{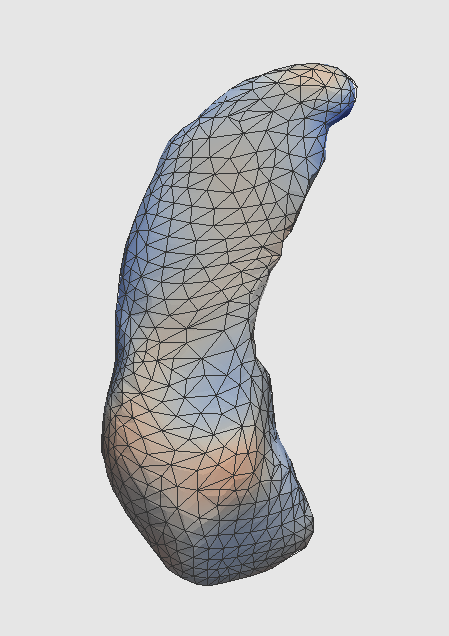}
\includegraphics[width=0.15\linewidth]{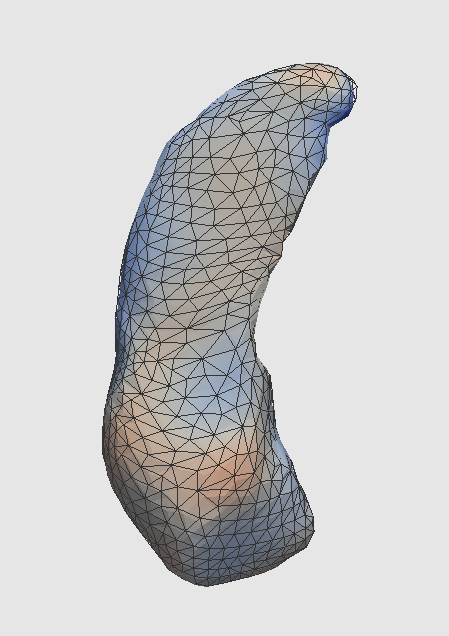}
\includegraphics[width=0.15\linewidth]{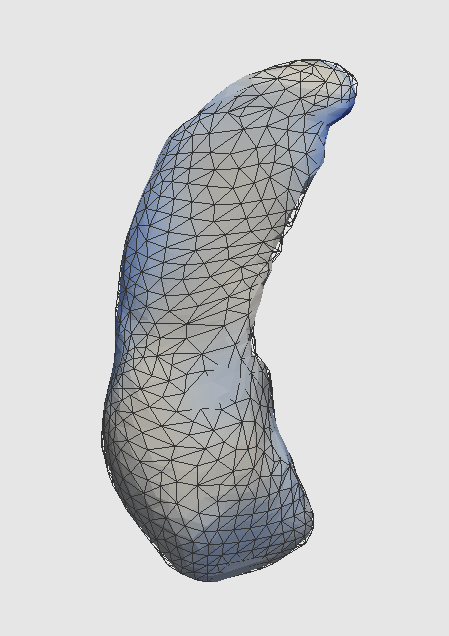}
\includegraphics[trim=6cm 5cm 6cm 5cm, clip,width=0.065\linewidth]{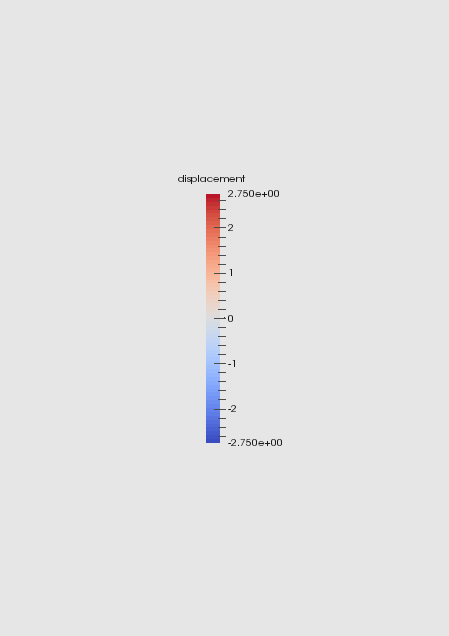}\\
\caption{\label{fig:1} Follow-up image is shown in wireframe in all views. From left to right:  baseline; LDDMM surface registration (no constraint); registration with non-increasing volume constraint; registration with atrophy constraint} 
\end{figure}

\bibliographystyle{plain}
\bibliography{refShape}

\end{document}